\documentclass[a4paper,12pt]{amsart}
\usepackage{amsmath}
\usepackage{amsthm}
\usepackage{amscd}
\usepackage{amssymb}
\usepackage{amsfonts}
\usepackage{latexsym}
\usepackage[mathscr]{eucal}

\newtheorem{thm}{Theorem}[section]
\newtheorem{prop}[thm]{Proposition}
\newtheorem{lemma}[thm]{Lemma}

\begin{document}

\title{ Systems of two subspaces in a Hilbert space}

\author{Masatoshi Enomoto}    
\address[Masatoshi Enomoto]{Koshien University, Takarazuka, Hyogo 665-0006, 
Japan}
\email{enomotoma@hotmail.co.jp}

\author{Yasuo Watatani}
\address[Yasuo Watatani]{Department of Mathematical Sciences, 
Kyushu University,  
Fukuoka, Japan}
\email{watatani@math.kyushu-u.ac.jp}
\maketitle
\begin{abstract}
We study two subspace systems in a separable 
infinite-dimensional Hilbert space up to (bounded) isomorphism. One of the main result of this paper is
the following:
Isomorphism classes of two subspace
 systems given by graphs of bounded 
operators
are determined by 
unitarily equivalent classes of the operator ranges and the nullity of the original bounded operators giving graphs.
We construct several non-isomorphic examples of two subspace systems  in an 
infinite-dimensional Hilbert space.


\medskip\par\noindent
AMS SUBJECT CLASSIFICATION: 46C07, 47A15, 16G20, 16G60.

\end{abstract}

\section{Introduction}

Let $E_{1}$ and $E_{2}$ 
be two  closed subspaces in a Hilbert space
$H$, then 
we say that 
$(H;E_{1},E_{2})$
is a two subspace system in $H$
or a system of two subspaces in $H$.
Let 
$(L;F_{1},F_{2})$
be another two subspace system in $L$.
We say that $(H;E_{1},E_{2})$
and 
$(L;F_{1},F_{2})$ are unitarily (resp.boundedly, algebraically) isomorphic 
if there exists a unitary operator (resp. bounded invertible operator,
invertible operator) $V$ of $H$ to $L$ such that 
$V(E_{1})=F_{1}$
and
$V(E_{2})=F_{2}$. 
Unitary isomorphism classes of  two subspace systems are studied  by 
many authors (cf.H.Araki\cite{A}, C.Davis\cite{Da}, J.Dixmier\cite{D}, 
P.Halmos\cite{H}, M.Stone\cite{S} etc.).

It is easy to see that 
two subspace systems
 $(H;E_{1},E_{2})$
and
$(L;F_{1},F_{2})$
are algebraically isomorphic
if and only if 
${\rm Hdim}(E_{1}\cap E_{2}) ={\rm Hdim}(F_{1}\cap F_{2})$,
${\rm Hdim}(E_1/(E_{1}\cap E_{2})) = 
{\rm Hdim}(F_1/(F_{1}\cap F_{2}))$,  
${\rm Hdim}(E_2/(E_{1}\cap E_{2})) 
= {\rm Hdim}(F_2/(F_{1}\cap F_{2}))$ 
and 
${\rm Hdim}(H/(E_{1} + E_{2})) = 
{\rm Hdim}(L/(F_{1} +  F_{2}))$, 
where ${\rm Hdim(}K)$ is a Hamel dimension of a vector space $K$. 
For a Hilbert space, we denote by $\dim H$ the Hilbert space dimension of $H$, 
that is, the cardinality of an  
orthonormal basis (or a complete orthonormal system) of $H$.

These three types of isomorphisms
(unitary isomorphisms, bounded isomorphisms and algebraic isomorphisms)
are different each other.
Unitary isomorphisms and bounded isomorphisms of two subspace systems 
are distinguished by angles.
Bounded isomorhisms and 
algebraic isomorhisms
of two subspace systems
are also distinguished. For example, 
put $a_{n}=1/n$ and 
and  $b_{n}=1/n^2$.
Let 
$A$ be the diagonal operator with diagonals  $(a_{n})_n$
and
$B$ be the diagonal operator with diagonals  $(b_{n})_n$
on $K=\ell^2(\mathbb{N})$. Put $H=K\oplus K.$
Then 
two subspace systems 
$(H;K\oplus 0, graph(A))$ 
and 
$(H;K\oplus 0, graph(B))$ 
are algebraically  isomorphic,
but not boundedly isomorphic,
since $A$ and $B$ belong to different Schatten classes.
Bounded isomorphisms of systems of two subspaces 
have not been studied extensively compared with unitary isomorphisms.

In this paper 
we study two subspace systems up to  bounded isomorphism.
For this purpose, it is crucially important to investigate operator ranges.
We recall an important paper \cite{FW}  by Fillmore-Williams, 
which studies operator ranges.

One of the main result of this paper is
the following:
Isomorphism classes of two subspace
 systems given by graphs of bounded 
operators
are determined by 
unitarily equivalent classes of the operator ranges and the nullity of the original bounded operators giving graphs.
We describe a relation among derived three subspaces associated with two subspaces, 
$A_{2}$-Dynkin quiver
 and operator ranges.
We give several examples of two subspace systems.

The classification  problem of n subspaces in a Hilbert space 
up to unitary isomorphism arises naturally.
But the problem for $n\geq 3$ is 
*-wild in the sense of S.Kruglyak and Y.Samoilenko\cite{KS}
and extremely difficult.
See also S. Kruglyak, V. Rabanovich and Y. Samoilenko\cite{KRS},Y. Moskaleva and Y. Samoilenko\cite{MS} 
and Sunder\cite{Su} for the study of n subspaces .
We study three subspaces \cite{EW3} and n subspaces \cite{EW1} up to bounded
isomorphism which is weaker than unitary isomorphism.
We should remark that in our former papers we just called  "isomorphism"
for "bounded isomorphism". 
 
This work was supported by JSPS KAKENHI Grant Numbers 
JP 23654053 and JP 17K18739.
This work was also supported by the Research Institute for Mathematical Sciences,
a Joint Usage/Research Center located in Kyoto University.

\section{Two subspace systems given by graphs of operators}

In \cite{EW1},we studied  
several subspaces
in an infinite dimensional Hilbert space 
in general.
In this paper,
we classify   two subspaces
up to bounded
isomorphism
in a separable  Hilbert space.

\parskip=8pt
\noindent
{\bf Definition.}
For Hilbert spaces $H_{1}$ and $H_{2}$,
we denote by $B(H_{1},H_{2})$
 the set of bounded operators of 
$H_{1}$ to $H_{2}$.
For $H_{1}=H_{2}=H$, we denote by $B(H)$ the algebra of bounded operators 
on a  Hilbert space $H$ .
An operator range  in a Hilbert space $H$ is a linear subspace of $H$ that
is the range of some bounded operator on $H$.
If a vector space $\mathscr{R}$ is the range of a closed operator on $H$,
then $\mathscr{R}$ is an operator range.
Operator ranges $\mathscr{R}$ and $\mathscr{S}$ in $H$ are 
called similar,
if there is an invertible operator $T\in B(H)$
such that
$\mathscr{S}=T \mathscr{R}$,
and
unitarily equivalent,
if $T$ can be taken to be unitary.
Operator ranges are similar if and only if
they are unitarily equivalent (cf. \cite[Theorem 3.1]{FW}).
Let  $T$ be a densely defined closed operator on $H$. Then we denote by $Dom(T)$ the domain of $T$, by $ran T$ the set $T(H)$ and
by $ker T$ the set $T^{-1}(\{0\})$.
We denote by $C(H)$ the set of compact operators on a Hilbert space $H$. 
For vectors $x$ and $y$ in $H$, 
the symbol $x\otimes y$
represents an operator on $H$ such that 
$(x\otimes y)z=<z,y>x$
for $z\in H$.

Any compact operator $T$ on a Hilbert space $H$ has a form 
$T=\sum_{j=1}^{N}\mu_{j}(T)v_{j}\otimes u_{j} $
where $N=rank(T)$,
$\mu_{1}(T)\geq \mu_{2}(T)\geq \cdots \geq \mu_{N}(T)$
(or $\geq \cdots$ if $N=\infty$),
$(u_j)_j, (v_j)_j$
are orthogonal
families and 
$\mu_{j}(T)$
is 
the j-th eigenvalues
of 
$\vert T\vert.$ 
Using min-max principle we have 
the following known fact.
For any $T\in C(H)$
and $X\in B(H)$,
$\mu_{j}(T)=\mu_{j}(T^*)$,
$\mu_{j}(XT)\leq \vert\vert X\vert\vert\mu_{j}(T)$
and 
$\mu_{j}(TX)\leq \vert\vert X\vert\vert\mu_{j}(T)$.

The (extended) Schatten class for $\alpha>0$
is 
$$
C^{\alpha}(H):=\{T\in B(H);Tr(\vert T\vert^{\alpha})<\infty\},
$$
see MacCarthy[M] for the case that $0<\alpha<1$. 
For $\alpha>0$, we say that a sequence $(a_{n})_n$ of complex numbers 
is in $\ell^{\alpha}(\mathbb{N}$) 
if 
$\sum\vert a_{n}\vert^{\alpha}<\infty $. 

\noindent
{\bf Definition.} Let $H$ be a Hilbert space and 
$E_1, \dots E_n$ be $n$ closed subspaces 
in $H$.  Then we say that  ${\mathcal S} = (H;E_1, \dots , E_n)$  
is a system of $n$-subspaces in $H$ or an $n$-subspace system in $H$. 
Let ${\mathcal T} = (K;F_1, \dots , F_n)$  
be  another system of $n$-subspaces in a Hilbert space $K$. Then  
$\varphi : {\mathcal S} \rightarrow {\mathcal T}$ is called a bounded
homomorphism if $\varphi : H \rightarrow K$ is a bounded linear 
operator satisfying that  
$\varphi(E_i) \subset F_i$ for $i = 1,\dots ,n$. And 
$\varphi : {\mathcal S} \rightarrow {\mathcal T}$
is called a bounded isomorphism if $\varphi : H \rightarrow K$ is 
a bounded invertible linear 
operator satisfying that  
$\varphi(E_i) = F_i$ for $i = 1,\dots ,n$. 
We say that systems ${\mathcal S}$ and ${\mathcal T}$ are 
{\it bounded isomorphic} if there is a bounded isomorphism  
$\varphi : {\mathcal S} \rightarrow {\mathcal T}$.

\noindent
{\bf Definition.} Let $A: K_1 \rightarrow K_2$ be a closed operator of a  
Hilbert space $K_1$ to a Hilbert space $K_2$. Let $H = K_1 \oplus K_2$. 
Then the two subspace system $(H;K_1\oplus 0,graph(A))$ is said to be 
given by a $graph(A)$ of the operator $A$.

\noindent
{\bf Definition.}
Let $(H;E_{1},E_{2})$ be a two subspace system.
Then we call the three subspace system 
$(H;E_{1},E_{1}^{\perp},E_{2})$ is 
the derived three subspace system (or the derived three subspaces)
of $(H;E_{1},E_{2})$.

In this paper we mainly discuss 
isomorphisms by bounded invertible operators between $n$ subspace systems.  

The following known fact is useful to study bounded isomorphisms.
   
\begin{lemma} {\rm (}\cite[Lemma2.1]{EW1}{\rm)}
 Let $H$ be a Hilbert space and $H_1$ and $H_2$ be 
two closed subspaces of $H$.  Then the following are equivalent:
\begin{enumerate}
\item $H = H_1 + H_2$ \ and  \ $H_1 \cap H_2 = 0$.
\item There exists a closed subspace $M \subset H$ such that
$(H;H_1,H_2)$ is boundedly isomorphic to  $(H;M,M^{\perp})$
\item There exists an idempotent $P \in B(H)$ such that 
$H_1 =ran P$ and $H_2 = ran(1-P)$. 
\end{enumerate}
\label{lemma:decompose}
\end{lemma}

The following result is well known as in P.Halmos \cite{H}.
Let $M$ and $N$ be closed subspaces of a Hilbert space 
$H$. Then $M$ and $N$ are in generic position if
$$M\cap N,M\cap N^{\perp},M^{\perp}\cap N,M^{\perp}\cap N^{\perp}$$
are zero. For any such pair $M,N$,
there exist a Hilbert space $K$ and a closed linear operator $T$ having domain and range dense in $K$
and zero kernel, such that 
a unitary operator of $H$ onto
$K\oplus K$ carries $M$ to
$K\oplus 0$ and carries $N$ to the graph of $T$.
The linear operator $T$ can be chosen self-adjoint and positive, and if it is chosen so,
then it is unique up to 
unitary equivalence.

In bounded isomorphisms case,
the situation is completely different.
Let $K=\mathbb{C}^2$
and $
T_{1}
=
\begin{pmatrix}
1 & 0 \\
0 & 1/2
\end{pmatrix}
,
T_{2}=\begin{pmatrix}
1 & 0 \\
0 & 1/3
\end{pmatrix}
$.
Then
$T_{1}$ and
$T_{2}$ 
are not similar, but the  two systems
$(H;K\oplus 0,graph(T_{1}))$
and
$(H;K\oplus 0,graph(T_{2}))$
are boundedly isomorphic.

\noindent
{\bf Example.} 
Consider a sequence $(\theta_{n})_n$ such that $0<\theta_{n}<\pi/2$ and 
$\theta_{n}\to 0 (n\to\infty)$.
Let $C$(resp.$S$) be the diagonal operator
with diagonals
$(\cos\theta_{n})_n$
 (resp.
$(\sin\theta_{n})_n)$.
Let $K=\ell^2(\mathbb{N})$
and
$H=K\oplus K$.
Let $E_{1}=K\oplus 0$
and $E_{2}=ran
\begin{pmatrix}
C^{2} & CS \\
CS & S^{2}
\end{pmatrix}
.$
Then
$(H;E_{1},E_{2})$
is not boundedly isomorphic to 
$(K\oplus K;K\oplus 0,0\oplus K)$.
In fact, if 
$(H;E_{1},E_{2})$
were isomorphic to 
$(K\oplus K;K\oplus 0,0\oplus K)$,
then 
$E_{1}+E_{2}$ must be closed since 
$K\oplus K=K\oplus 0+0\oplus K$
is closed, but 
$
E_{1}+E_{2}=K\oplus ran S
$ is not closed because 
the range of $S$ is not closed. 
This is a contradiction.

We need to  recall  Hilbert representations of 
quivers studied in \cite{EW2}.

\noindent  
{\bf Definition.} 
A quiver $\Gamma=(V,E,s,r)$ is a quadruple consisting of 
the set $V$ of vertices, the set $E$ of arrows, 
and two maps $s, r : E \rightarrow V$, which 
associate with each arrow $\alpha \in E$ its 
support $s(\alpha)$ and range  $r(\alpha)$. We 
sometimes denote by $\alpha : x \rightarrow y$ 
an arrow with $x = s(\alpha)$ and $y = r(\alpha)$. 
Thus a quiver is just a directed graph. We 
denote by $|\Gamma|$ 
the underlying undirected graph of a quiver $\Gamma$. A quiver $\Gamma$ is said to be finite 
if both $V$ and $E$ are finite sets.

\noindent  
{\bf Definition.}
Let $\Gamma=(V,E,s,r)$ be a finite quiver. We say
that $(H,f)$ is a  {\it Hilbert representation} of $\Gamma$ 
if $H=(H_{v})_{v\in V}$  is a family of  Hilbert spaces 
and $f=(f_{\alpha})_{\alpha \in E}$ is a family of
 bounded linear operators $f_{\alpha} : 
H_{s(\alpha)}\rightarrow H_{r(\alpha)}.$

\noindent  
{\bf Definition.} Let $\Gamma=(V,E,s,r)$ be a finite quiver. 
Let $(H,f)$ and $(K,g)$ be Hilbert representations of $\Gamma.$ 
A {\it bounded homomorphism} $T : (H,f) \rightarrow (K,g)$  is a 
family $T =(T_{v})_{v\in V}$ of bounded operators 
$T_v : H_v \rightarrow K_v$ satisfying,  
for any arrow $\alpha \in E$ 
$$
T_{r(\alpha)}f_{\alpha}=g_{\alpha}T_{s(\alpha)}. 
$$

Let $\Gamma=(V,E,s,r)$ be a finite quiver and 
$(H,f)$, $(K,g)$  be Hilbert representations of $\Gamma.$ 
We say that
$(H,f)$ and $(K,g)$ are {\it boundedly isomorphic}, denoted by  
$(H,f)\simeq(K,g)$, 
if there exists a bounded isomorphism $\varphi : (H,f) \rightarrow (K,g)$, 
that is, there exists a family  
$\varphi=(\varphi_{v})_{v\in V}$ of bounded invertible operators
$\varphi_{v}\in B(H_{v},K_{v})$ such that, for any arrow 
$\alpha \in E$, 
$$\varphi_{r(\alpha)}f_{\alpha}=g_{\alpha
}\varphi_{s(\alpha)}.$$

We say that $\Gamma$ is  the $A_{2}$- Dynkin quiver
if $\Gamma=(V,E,s,t)$ is an oriented graph such that 
the vertex set  of 
$\Gamma$ is $V=\{1,2\}$,
the arrow set of $\Gamma$ is $E=\{\alpha\}$
 with 
\[
\circ_1 \overset{\alpha}\longrightarrow 
\circ_2  
\]

A Hilbert representation $(H,f)$ of the $A_{2}$-Dynkin quiver $\Gamma$ is called a Hilbert representation constructed by an operator $T: H_1 \rightarrow H_2$ if
$H_{s(\alpha)}$ is a Hilbert space $H_{1}$,
$H_{r(\alpha)}$ is a Hilbert space $H_{2}$ and
$f_{\alpha}=T$.

We mainly study two subspace systems which are
given by graphs of bounded operators.
The following is  the main theorem of the paper. 

\begin{thm}
Let $K_{1},K_{2}$ be Hilbert spaces and 
$T$, $T^{\prime}$ be in $B(K_{1},K_{2}).$
We put $H=K_{1}\oplus K_{2}$.
Then the following are equivalent:\\
{\rm (1)}$(H;K_{1}\oplus 0,graph(T))$ is 
boundedly isomorphic to
 $(H;K_{1}\oplus 0,graph(T^{\prime}))$.\\
{\rm(2)}
Derived three subspace systems $(H;K_{1}\oplus 0,0\oplus K_{2},graph(T))$
and $(H;K_{1}\oplus 0,0\oplus K_{2},graph(T^{\prime}))$ are boundedly isomorphic.\\
{\rm(3)}
Hilbert representations constructed by
$T:K_{1}\to K_{2}$
and
$T^{\prime}:K_{1}\to K_{2}$
are boundedly isomorphic as  Hilbert representations
 of the $A_{2}$-Dynkin quiver.\\
{\rm(4)} Operator ranges $ran T$ and 
$ran T^{\prime}$
are unitarily equivalent and \\
 $\dim ker T=
\dim ker T^{\prime}$.\\
{\rm(5)}
There exist invertible operators $L\in B(K_{2})$ 
and $M\in B(K_{1})$
such $T=LT^{\prime}M$.
\label{pro:3sub}
\end{thm}

\begin{proof}

(1)$\Rightarrow$ (4):
Assume that (1) holds.
Then there exists an invertible operator $S\in B(K_{1}\oplus K_{2})$ such that
$S(K_{1}\oplus 0)=K_{1}\oplus 0$
and 
$S(graph (T))=graph (T^{\prime})$.
Hence there exist operators $A\in B(K_{1}),B\in B(K_{2}),C\in B(K_{2},K_{1})$
such that 
$S=
\begin{pmatrix}
A & C \\
0 & B
\end{pmatrix}
$.
Since $S(K_{1}\oplus 0) = K_{1}\oplus 0$,
$A$ is surjective.
If $Ax=0$ for $\in K_{1}$,
then 
$
\begin{pmatrix}
A & C \\
0 & B 
\end{pmatrix}
\begin{pmatrix}
x \\
0 
\end{pmatrix}
=
\begin{pmatrix}
0\\
0 
\end{pmatrix}$
.
Since $S$ is invertible, 
$x=0$.
Thus $A$ is injective.
Therefore $A$ is bounded invertible.
Since 
$S$ is invertible, $B$ is surjective.
Assume that 
$By=0$ for $y\in K_{2}$.
We put $x=-A^{-1}Cy$.
Then 
$$
\begin{pmatrix}
A & C \\
0 & B 
\end{pmatrix}
\begin{pmatrix}
x \\
y
\end{pmatrix}
=
\begin{pmatrix}
A & C \\
0 & B 
\end{pmatrix}
\begin{pmatrix}
-A^{-1}Cy\\
y
\end{pmatrix}
=
\begin{pmatrix}
0\\
0
\end{pmatrix}
.
$$
Hence $y=0$.
Thus $B$ is injective and $B$ is bounded invertible.
Since 
$S(graph (T)) \subset graph (T^{\prime})$,
for any $x_{1}\in K_{1}$,
we have
$$
\begin{pmatrix}
A & C \\
0 & B 
\end{pmatrix}
\begin{pmatrix}
x_{1}\\
Tx_{1}
\end{pmatrix}
=
\begin{pmatrix}
(A+CT)x_{1}\\
BTx_{1}
\end{pmatrix}
=
\begin{pmatrix}
x_{2}\\
T^{\prime}x_{2}
\end{pmatrix}
\text{for some }x_{2}\in K_{1}
.$$
So $BT=T^{\prime}(A+CT)$ .
Since 
$S(graph (T)) \supset graph (T^{\prime})$,
$A+CT$ is surjective.
Hence $B(ran T)=ran T^{\prime}.$
By Fillmore and Williams \cite[Theorem 3.1]{FW},
$ran T$ is unitarily equivalent to $ran T^{\prime}.$
We shall show that  $\dim ker T$ $=
\dim ker T^{\prime}$.
Since 
$$
(K_{1}\oplus 0)
\cap (graph(T))=\ker T \oplus  0,
$$
the bounded isomorphism $S
=\begin{pmatrix}
A & C \\
0 & B
\end{pmatrix}$ gives
$S(\ker T \oplus  0)
=\ker T^{\prime} \oplus  0$.
Hence $\dim ker T=
\dim ker T^{\prime}$.
Thus (4) holds.\\
\noindent
(4)$\Leftrightarrow$ (5): This follows from a result \cite[Theorem 3.4]{FW}.\\
\noindent
(5)$\Rightarrow$ (3):It is trivial. \\
\noindent
(3)$\Rightarrow$ (2): Assume that (3) holds.
There exist bounded 
invertible operators 
$G_{1}\in B(K_{1})$ and $G_{2}\in B(K_{2})$
such that
$T^{\prime}G_{1}=G_{2}T.$
We put 
$S=
\begin{pmatrix}
G_{1} & 0 \\
0 & G_{2}
\end{pmatrix}.
$
Then 
$S$ is an invertible map on $K_{1}\oplus K_{2}.$

$
S(graph (T))=
\{\begin{pmatrix}
G_{1} & 0 \\
0 & G_{2}
\end{pmatrix}
\begin{pmatrix}
x_{1} \\
Tx_{1}
\end{pmatrix}
\vert x_{1}\in K_{1}\}
$

$=
\{\begin{pmatrix}
G_{1}x \\
G_{2}Tx_{1}
\end{pmatrix}
\vert x_{1}\in K_{1}\}
$
$=
\{\begin{pmatrix}
G_{1}x \\
T^{\prime}G_{1}x_{1}
\end{pmatrix}
\vert x_{1}\in K_{1}\}
.$

Since $ran  G_{1}=K_1$,
we have 
$S(graph (T))=graph (T^{\prime})$.
Thus (2) holds.\\
\noindent
(2)$\Rightarrow$ (1): It is trivial.
\end{proof}

\noindent  
{\bf Remark.} 
We see that $A+CT$ is also 
one to one as pointed out to us by R. Sato and Y. Ueda. 
Therefore we can directly show that (1) implies (5) without a result 
in \cite{FW}. 

\noindent  
{\bf Remark.} The theorem above does not hold if $T_1$ or $T_2$ is not 
bounded. 
Let $T$ be a densely defined closed operator
with the domain $Dom(T)$ of $T$ 
in a Hilbert space $K$ and $H=K\oplus K$.
Assume that $Dom(T)\ne K$.
Let $U$ be a bounded operator on $K$.
Then
 derived three subspace systems $\tilde{\mathcal{S}_1}
=(H;K\oplus 0,0\oplus K,graph(T))$ and 
$\tilde{\mathcal{S}_2}=(H;K\oplus 0,0\oplus K,graph(U))$
are not boundedly isomorphic although 
$\mathcal{S}_{1}=(H;K\oplus 0,graph(T))$
and
$\mathcal{S}_{2}=(H;K\oplus 0,graph(U))$
are boundedly isomorphic to $(H;K\oplus 0,0\oplus K)$ by 
Lemma \ref{lemma:decompose}.  On the contrary, suppose that 
$\tilde{\mathcal{S}_1}$ were  boundedly isomorphic to  $\tilde{\mathcal{S}_2}$,
then
there exists 
an invertible operator $W$ 
on $H$ such that
$W(K\oplus 0)=K\oplus 0,
W(0\oplus K)=0\oplus K$.
The operator $W$ has the form
$W=
\begin{pmatrix}
A&0 \\
0&B
\end{pmatrix}
$,where $A$ and $B$ are invertible.
Since $W(graph(U))=graph(T)$,
$AK=Dom(T)$. Hence $Dom(T)=K$.
This is a contradiction.
Hence 
 $\tilde{\mathcal{S}_1}$ is not boundedly isomorphic to  $\tilde{\mathcal{S}_2}$.

\noindent  
{\bf Remark.} 
If $T$ is a normal operator,
then $ker T=ker T^*=(ran T)^{\perp}$.
Therefore,
if $T_{1}$ and $T_{2}$ are normal,
then the condition (4) is equivalent to
that
$ran T_{1}$ and $ran T_{2}$
are unitarily equivalent.


\section{Examples  of non-isomorphic two subspace systems}
At first we consider examples of two subspace systems given by 
graphs of compact operators. 

\begin{prop}
Let $A$ and $B$ be compact positive operators on a Hilbert space $K$.
We may assume that
there exist orthonormal systems $(x_{n})_n$ and 
$(y_n)_n$ in $K$
such that 
$A=
\sum_n \mu_{n}(A)x_{n}\otimes x_{n}$,
$B=
\sum_n \mu_{n}(B)y_{n}\otimes y_{n}$.
Also assume that
$\mu_{n}(A)\ne 0$ and $\mu_{n}(B)\ne 0$
for any $n\in \mathbb{N}$.
Then the following {\rm(i)} and {\rm(ii)} are equivalent:\\
\noindent
{\rm(i)}
$\dim ker A=\dim ker B$  and there exist positive numbers
$\gamma_{1},\gamma_{2}$
such that,  for any $n \in {\mathbb N}$
$$
\gamma_{1}\mu_{n}(B)
\leq \mu_{n}(A)
\leq \gamma_{2}\mu_{n}(B)
.$$
\noindent
{\rm(ii)}
$(H;K\oplus 0,graph(A))$ is boundely
isomorphic to
 $(H;K\oplus 0,graph(B))$.

\label{prop:mu}
\end{prop}
\begin{proof}

(i)$\Rightarrow$ (ii):Assume (i). 
Since $\dim ker A=\dim ker B$,
there exists a unitary operator $U$ such that 
$Uy_{n}=x_{n}$
for $n\in \mathbb{N}$.
Then we fave that 
$$
UBU^*=
\sum_n \mu_{n}(B)Uy_{n}\otimes Uy_{n}
= \sum_n \mu_{n}(B)x_{n}\otimes x_{n}.
$$
The positive sequence 
$(\nu_{n})_n:= (\mu_{n}(A)/\mu_{n}(B))_n $ is bounded and bounded below 
by the assumption (i).  
Take an orthonormal basis   
$\{z_{n}\}$ of $ker A$ and 
define  a bounded invertible diagonal operator $Z:= \sum\nu_{n}x_{n}
\otimes x_{n}+
 \sum z_{n}
\otimes z_{n}
$.
Then we have
$$
UBU^*Z=
A=
\sum\mu_{n}(A)x_{n}\otimes x_{n}
.$$
By Theorem \ref{pro:3sub}, we have (ii).  

(ii)$\Rightarrow$ (i):
Assume (ii). 
Then there exist bounded invertible operators $C$ and $D$ such that
$A=CBD$ by Theorem \ref{pro:3sub}. 
Hence 
$$\mu_{n}(A)=\mu_{n}(CBD)\leq \vert\vert C\vert\vert
\vert\vert D\vert\vert \mu_{n}(B).$$
and
$$\mu_{n}(B)=\mu_{n}(C^{-1}AD^{-1})\leq \vert\vert C^{-1}\vert\vert
\vert\vert D^{-1}\vert\vert \mu_{n}(A).$$
Put
$\gamma_{1}=\vert\vert C^{-1}\vert\vert^{-1}
\vert\vert D^{-1}\vert\vert^{-1}
$ and
$\gamma_{2}=\vert\vert C\vert\vert
\vert\vert D\vert\vert
$.
Then we have 
$$
\gamma_{1}\mu_{n}(B)
\leq \mu_{n}(A)
\leq \gamma_{2}\mu_{n}(B)
.$$
Since $A=CBD$,
$ker A=D^{-1}(ker B)$.
Hence 
$\dim kerA= \dim ker B$.

Thus (i) holds.

\end{proof}

We study 
 two subspace systems given by graphs of Schatten class operators.
We shall consider an invariant for such two subspace systems.

Let $T$ be a Schatten class operator on a Hilbert space $K$.
We put
$$Sh(T):=inf\{\alpha>0:T\in C^{\alpha}(K)\}.$$

For example, 
let $T$ be the diagonal operator with diagnals $(1/n^s)_n$ for $s > 0$, 
then $Sh(T) = 1/s$.

\begin{prop}
Let $T_{1},T_{2}$ be 
Schatten class operators on a Hilbert space $K$.
If
$(H;K\oplus 0,graph (T_{1}))$
is boundedly isomorphic to 
$(H;K\oplus 0,graph (T_{2}))$, 
then 
$Sh(T_{1})=Sh(T_{2})$.
But the converse does not hold.
\label{prop:Sh}
\end{prop}
\begin{proof}
For $1\leq \alpha$,
it is known that $C^{\alpha}(K)$ 
is a  ideal in $B(K)$.
For $0<\alpha<1,$
it seems that this fact is not well known. 
So we shall give a proof for completeness.
Heinz-inequality says that
for positive operators $A, B\in B(K)$ and $0\leq p\leq 1$, if $A\leq B$, then 
$A^{p}\leq B^{p}$. 
Using this inequality, we have
$$
\vert BT\vert^{\alpha}
=(\vert BT\vert^2)^{{\alpha}/2}
=(T^*B^*BT)^{{\alpha}/2}
\leq(T^*\vert\vert B
\vert\vert^2T)^{{\alpha}/2}\\
=\vert\vert B
\vert\vert^{\alpha}(\vert T\vert^{{\alpha}}).
$$
Therefore 
if $T\in C^{\alpha}(K)$,then
 $BT\in C^{\alpha}(K)$ for $\alpha>0.$
If $Tr(\vert T\vert^{\alpha})<\infty$
for $0< \alpha$, 
then
$$Tr(\vert T\vert^{\alpha})=
Tr(\vert T^*\vert^{\alpha}).$$
From this,
if $T\in C^{\alpha}(K)$,then
 $TB\in C^{\alpha}(K)$ for $\alpha>0.$
Therefore $ C^{\alpha}(K)$ is an ideal in $B(K)$,  since 
it is known that for $0< \alpha<1$, $ C^{\alpha}(K)$ is a linear space 
as in ([M],Theorem 2.8]), for example.

Suppose that 
 $(H;K\oplus 0,graph (T_{1}))$
is boundedly isomorphic to 
$(H;K\oplus 0,$ $graph (T_{2}))$. 
By  Theorem \ref{pro:3sub},
there exist bounded invertible operators $L,M\in B(K)$ such that 
$T_{1}=LT_{2}M.$ Since $ C^{\alpha}(K)$ is an ideal of $B(K)$ for any 
$\alpha$, $T_1 \in C^{\alpha}(K)$ if and only if $T_2 \in C^{\alpha}(K)$. 
Therefore $Sh(T_{1})=Sh(T_{2})$. 

But the converse does not hold. In fact, let $T_{1}$ be the diagonal operator
with diagonals $(1/n)_n$
and $T_{2}$
the diagonal operator
with
diagonals $(1/\{(n+1)\log(n+1)\})_n$.
Then $Sh(T_1)=Sh(T_2)$
and
two subspace systems 
$(H;K\oplus 0,graph (T_{1}))$ and
$(H;K\oplus 0,graph (T_{2}))$ are not
boundedly isomorphic
by Proposition \ref{prop:mu}.
\end{proof}

\noindent
{\bf{Example.}} Let K be a Hilbert space 
with a basis $(e_n)_n$.  
Let 
$A=\Sigma_n 1/n(e_{n+1}\otimes e_n)$
and
$B=\Sigma_n 1/n(e_n\otimes e_n)$.
Then $A$ and $B$ are Schatten class operators.
Clearly $Sh(A)=Sh(B)=1$.
But the two subspace systems given by the graphs of $A$ and $B$ 
are not isomorphic, 
since $ran A$ is not unitarily equivalent to $ran B$.

We note that 
Schatten class operators do not exhaust all compact operators.
Consider a diagonal operator $T$  with
 diagonals $(1/\log(n+1))_{n}$. Then $T$ is a compact operator but $T$ does not belong to any Schatten class operator.
The next proposition can be applied for such an operator.

\begin{prop}
Let $s,t$ be positive numbers and $s\ne t$.
Let
$c_{n}$ be a decreasing sequence of positive numbers
with $\lim_{n\to \infty} c_{n}=0$. 
Put $K=\ell^2(\mathbb{N})$ and $H=K\oplus K$.
Let 
$A$(resp.$B$) be a diagonal operator on $K$ with diagonals 
$(c_{n}^s)_{n}$
{\rm (}resp. 
$(c_{n}^t)_{n}${\rm )}.
Then $(H;K\oplus 0,graph(A))$ is 
not boundedly isomorphic to
 $(H;K\oplus 0,graph(B))$.

\label{prop:c(n)}
\end{prop}
\begin{proof}
On the contrary, suppose that 
$(H;K\oplus 0,graph(A))$ were boundedly isomorphic to
 $(H;K\oplus 0,graph(B)). $ 
Then by Proposition \ref{prop:mu},
there exist positive numbers
$\gamma_{1},\gamma_{2}$
such that
$$
\gamma_{1}\mu_{n}(B)
\leq \mu_{n}(A)
\leq \gamma_{2}\mu_{n}(B)
.$$

Hence
$$
\gamma_{1}c_{n}^t
\leq c_{n}^s
\leq \gamma_{2}c_{n}^t
.$$

If $s>t$,by 
$\gamma_{1}
\leq c_{n}^{(s-t)},$
this is a contradiction.
If $s<t$,by 
 $c_{n}^{(s-t)}
\leq \gamma_{2}
,$
this is a contradiction.
This proves the  theorem. .

\end{proof}


\noindent
{\bf Example.} Let $K=L^{2}[0,1]$ and $s,t>0$.
Let 
$M_{x^s}$ be the multiplication operator 
on $K$ such that
$M_{x^s}f(x)=x^sf(x)$ for $f\in K$ and $x\in [0,1]$.
Then $ran M_{x^s}$ is unitarily equivalent to $ran M_{x^{t}}$.

In fact
we put
$U:K\to K$ by
$$(Uf)(x)=\sqrt{sx^{s-1}}f(x^{s})$$
for
$f\in K$. Then
$U$ is a unitary and
$UM_{x}=M_{x^s}U$.
Thus
$ran M_{x}$
and
$ran M_{x^s}$
are
unitarily equivalent.


Next we shall consider when two subspaces are 
algebraically isomorphic. The Hamel dimension of any infinite dimensional
separable Banach space is continuously infinite (cf.[L]). 
The Hamel dimension of an operator range in a separable Hilbert
space $K$ is finite or continuous, since, for $A\in B(K)$, 
$K/ker A$ is algeraically isomorphic to $ran A$.
For any non-closed operator range $\mathscr{R}$ in a separable 
Hilbert space $K$,
the Hamel dimension of $K/\mathscr{R} $ is continuous.
See, for example, \cite[page.274,Cor1]{FW}.
Let  $c_{0}$ be the vector space 
of sequences which converges to 0 and
let $c_{00}$ be the subspace of sequences with a finite support.
Clearly the Hamel dimension of $c_{oo}$ is countable. 
Thus $c_{oo}$ can not be an operator range  in $\ell^2(\mathbb N)$.

It is easy to see the following:

\begin{prop}Let $H$ and $L$ be Hilbert spaces. Then the following are 
equivalent. \\
{\rm(1)}
Two subspace systems
 $(H;E_{1},E_{2})$
and
$(L;F_{1},F_{2})$
are algebraically isomorphic. \\
{\rm(2)}
$${\rm Hdim}(E_{1}\cap E_{2}) ={\rm Hdim}(F_{1}\cap F_{2}),$$
$${\rm Hdim}(E_1/(E_{1}\cap E_{2})) = 
{\rm Hdim}(F_1/(F_{1}\cap F_{2})),$$  
$${\rm Hdim}(E_2/(E_{1}\cap E_{2})) 
= {\rm Hdim}(F_2/(F_{1}\cap F_{2}))$$ 
and 
$${\rm Hdim}(H/(E_{1} + E_{2})) = 
{\rm Hdim}(L/(F_{1} +  F_{2})). $$
\end{prop}

The following proposition is a direct consequence of 
the proposition above.  

\begin{prop}
Consider $(a_{n})_{n}\in \ell^{\infty}(\mathbb{N})$
and   $(b_{n})_{n}\in \ell^{\infty}(\mathbb{N})$
such that $a_{n}\ne0$ and $b_{n}\ne0$ for any 
$n\in \mathbb{N}$.
Let  $A$ and $B$  be
diagonal operators on $K=\ell^2(\mathbb{N})$ with  diagonals 
 $(a_{n})_n$ and $(b_{n})_{n}$ respectively.
Put $H=K\oplus K.$ Then the following hold. \\
{\rm (i)} If $ran A$ is closed and $ ran B$
is not closed, then two subspace systems 
$(H;K\oplus 0, graph(A))$ 
and 
$(H;K\oplus 0, graph(B))$ 
are not algebraically isomorphic.\\
{\rm (ii)} If $ran A$ and $ ran B$
are both closed or both non-closed, then two subspace systems 
$(H;K\oplus 0, graph(A))$ 
and 
$(H;K\oplus 0, graph(B))$ 
are algebraically isomorphic.
\end{prop}
\begin{proof}
(i) We assume that $ran A$ is closed
 and $ran B$ 
is non-closed.
Then we have that  $ran A=K$ and
$$
(K\oplus K)/(K\oplus ranA)=
(K\oplus K)/(K\oplus K)=0, 
$$ 
but 
$(K\oplus K)/(K\oplus ranB)\ne 0$.
So these two subspace systems are not algebraically isomorphic.\\
(ii)Consider the case that 
the operator ranges $ranA$ and $ran B$ 
are non-closed. Then 
$(K\oplus 0)\cap(graph A)=ker A \oplus 0 =0$
and
$(K\oplus 0)\cap(graph B)=ker B \oplus 0=0$.
And
$K\oplus 0$,$graph A$
and $graph B$
are all algebraically  isomorphic to $K$.
Moreover the  
Hamel dimensions of 
$(K\oplus K)/(K\oplus ranA) = K/ran A$ and 
$(K\oplus K)/(K\oplus ranB) = K/ran B$ are 
both continuous, because 
$ranA$ and $ran B$ are non-closed. 
Hence 
the two subspace systems 
$(H;K\oplus 0, graph(A))$ 
and 
$(H;K\oplus 0, graph(B))$ 
are algebraically isomorphic. 
Next consider the case that the operator ranges $ranA$ and $ran B$ 
are closed. Then $ranA = K = ran B$. Similar consideration implies the 
conclusion. 
\end{proof}
For example, let $a_n = 1/n$ and $b_n = 1/n^2$. Then 
two subspace systems 
$(H;K\oplus 0, graph(A))$ 
and 
$(H;K\oplus 0, graph(B))$ 
are algebraically isomorphic,
but not boundedly isomorphic. 

\noindent
{\bf Example.} Let 
$A$ be the diagonal operator with diagonals $(n^2)_{n}$
and $A^{\prime}$ be the diagonal operator with diagonals $(2)_{n}$
on $K=\ell^2(\mathbb{N})$.
We put $H=K\oplus K$.
Then
$(H;K\oplus 0,graph(A))$
and
$(H;K\oplus 0,graph (A^{\prime}))$
are boundedly isomorphic.
In fact this follows from lemma \ref{lemma:decompose}, 
since  $K\oplus 0 + graph (A)=K\oplus K$, 
$
K\oplus 0 + graph (A ^{\prime}) 
=K\oplus K. \\
(K\oplus 0)\cap graph (A)=0$ and 
$(K\oplus 0)\cap graph (A ^{\prime}) =0$. 

\noindent
{\bf Example.} Let 
$A$(resp. $A^{\prime}$, $C$ )be the diagonal operator 
with diagonals $(n^2)_n$, 
(resp. $(2)_{n}$, $(1/n^2)_{n}$)
on $K=\ell^2(\mathbb{N})$.
We put $H=K\oplus K$. 
Then
$(H\oplus H;H\oplus 0,graph (A\oplus C))$
and
$(H\oplus H;H\oplus  0,graph (A^{\prime}\oplus C))$
are boundedly isomorphic. In fact $(H;K\oplus 0,graph (A))$
and
$(H;K\oplus  0,graph (A^{\prime}))$
are boundedely isomorphic,
and $(H\oplus H;H\oplus  0,graph (A\oplus C))$
is boundedly isomorphic to 
$(H;K\oplus 0,graph(A))\oplus (H;K\oplus 0,graph(C))$.

We give a condition 
when two subspace systems
given by  graphs of  unbounded operators
are boundedly isomorphic.

\begin{prop} 
Let $T_{1},T_{2}$ be densely defined closed operators 
on a Hilbert space $K$ such that $T_{1}^{-1},T_{2}^{-1}$
are bounded operators.
If $\vert\vert T_{1}^{-1}-T_{2}^{-1}\vert\vert <1$, 
then two subspace systems 
$(H;K\oplus 0,graph(T_{1}))$
and
$(H;K\oplus 0,graph(T_{2}))$
are boundedly isomorphic.

\label{prop:inverse}
\end{prop}
\begin{proof}
Let $J(x,y):= (y,x)$ for $x,y\in K$.
Then $J$ gives an bounded isomorphism 
of 
$(H;K\oplus 0,graph(T))$
to
$(H;0\oplus K,graph(T^{-1}))$. 

We put
$\Phi=
\begin{pmatrix}
I & 0 \\
T_{2}^{-1}-T_{1}^{-1} & I
\end{pmatrix}$.
Since
$\vert\vert T_{2}^{-1}-T_{1}^{-1}\vert\vert <1$,
$\Phi$
is invertible
and 
$$\Phi(graph(T_{1}^{-1}))
=graph(T_{2}^{-1}) \text{ and }
\Phi(0\oplus K)
=0\oplus K.$$ 
Hence
$(H;0\oplus K,graph(T_{1}^{-1}))$
and
$(H;0\oplus K,graph(T_{2}^{-1}))$
are boundedly isomorphic.
Therefore two systems 
$(H;K\oplus 0,graph(T_{1}))$
and
$(H;K\oplus 0,graph(T_{2}))$
are boundedly isomorphic.
\end{proof}

We recall a useful lemma in \cite{FW}. 

\begin{lemma}\rm{\cite[Lemma 3.2]{FW}}
Let $A=\int_{0}^{M}\lambda dE_{\lambda}$ and
$B=\int_{0}^{N}\lambda dF_{\lambda}$ be positive operatos 
on a Hilbert space $H$.
Assume that $ran A =ran B$.
Then
there is a positive constant $K\geq 1$
such that
$$\dim E[\alpha,\beta]H\leq \dim F[\alpha/K,K\beta]H.$$
and 
$$\dim F[\alpha,\beta]H\leq \dim E[\alpha/K,K\beta]H.$$
whenever 
$0<\alpha\leq \beta.$

\label{FW}
\end{lemma}

The above lemma \ref{FW} enables us to consider the following examples.

\noindent
{\bf Example.} Let 
$A$ (res.$B$,$C$,$D$) be the diagonal operator with diagonals 
$(n^2)_{n}$ (resp.$(n^3)_{n}$, $(1/n^2)_{n}$, $(1/n^3)_{n}$)
on $K=\ell^2(\mathbb{N})$.
Then
$(H\oplus H;H\oplus 0,graph (A\oplus C))$
and
$(H\oplus H;H\oplus 0,graph (B\oplus D))$
are not boundedly isomorphic. 
In fact, let 
$A^{\prime}$ be the diagonal operator with diagonals $(2)_n$ on $K=\ell^2(\mathbb{N})$.
Then
$(H\oplus H;H\oplus 0,graph(A\oplus C))$
and
$(H\oplus H;H\oplus 0,graph(A^{\prime}\oplus C))$
are boundedly isomorphic, and 
$(H;K\oplus 0,graph(B\oplus D))$
and
$(H;K\oplus 0,graph(A^{\prime}\oplus D))$
are also boundedly isomorphic. 
Put $c_n = 1/n^2$ and $d_n = 1/n^3$. 
Let $E$ be a spectral measure for $A^{\prime}\oplus C$
and 
$F$ a spectral measure for $A^{\prime}\oplus D$. 

By lemma \ref{FW}, ,
there exists a positive constant $K\geq 1$
such that
\begin{align*}
{}&\dim(E[\alpha,1/K]H)={}^{\#}\{n\in {\mathbb N};\alpha\leq 
c_n\leq (1/K)\}\\
{}&\leq \dim(F[\alpha/K,1]H)=
{}^{\#}\{n\in {\mathbb N};\alpha/K\leq 
 d_{n}\leq 1\} \\
{}&={}^{\#}\{n\in {\mathbb N};\alpha\leq K 
 d_{n}\leq K\}.
\end{align*}
Put 
$$
m_0 := \min\{m\in {\mathbb N};c_m \leq 1/K \}
$$
For any $n \in  {\mathbb N}$, put $\alpha = c_{n + m_0}$. Then
$$
n+1 = {}^{\#}\{\ell \in {\mathbb N} ; \alpha \leq c_{\ell} \leq 1/K \} 
\leq  {}^{\#}\{\ell \in {\mathbb N} ; \alpha \leq Kd_{\ell} \leq K \}. 
$$
Hence $\alpha \leq Kd_{n+1} \leq Kd_n$. Therefore 
$c_{n + m_0} \leq Kd_n$. Thus for any $n \in {\mathbb N}$, 
we have that 
$$
n^3/(n+m_0)^2\leq K
$$
This implies a contradiction. Therefore 
these two subspaces are not boundedly isomorphic.

Next we give another example. 

\noindent
{\bf Example.} 
Let $T\in B(L^2[2,3])$ be a multiplication operator 
defined by
$(Tf)(t)=tf(t)$ for $f\in L^2[2,3]$.
Let 
$C$ be the diagonal operator with diagonals $(1/n^2)_{n}$
and $D$ be the diagonal operator with diagonals $(1/n^3)_{n}$
on $K=\ell^2(\mathbb{N})$.
We put $H=K\oplus K$.
Then
operators $T\oplus C$ and $T\oplus D$
have continuous spectrum
and 
$(H\oplus H;H\oplus 0,graph (T\oplus C))$
and
$(H\oplus H;H\oplus 0,graph (T\oplus D))$
are not boundedly isomorphic. 

Use lemma \ref{FW} similarly for the intervals not containing the interval 
$[2,3]$.


\end{document}